\documentclass[12pt]{amsart}
\usepackage{color,latexsym,amsthm,amsmath,amssymb,path,url}
\usepackage[utf8]{inputenc}
\usepackage[english]{babel}
\newcommand{\ignore}[1]{}

\renewcommand{\ge}{\geqslant}
\renewcommand{\le}{\leqslant}

\newtheorem{theorem}{Theorem}
\newtheorem{lemma}[theorem]{Lemma}

\newtheorem{statement}[theorem]{Statement}

\newtheorem{rem}[theorem]{Remark}

\newcommand{\Proof}[1]
        {
        \noindent
        \emph{Proof #1.}~
        }
\newsavebox{\smallProofsym}                     
\savebox{\smallProofsym}                        %
        {
        \begin{picture}(7,7)                    %
        \put(0,0){\framebox(6,6){}}             %
        \put(0,2){\framebox(4,4){}}             %
        \end{picture}                           %
        }                                       %
\newcommand{\smalleop}[1]
        {
        \mbox{} \hfill #1~~\usebox{\smallProofsym}\!\!\!\!\!\!\
        }

\usepackage{graphicx,psfrag}
\usepackage[rflt]{floatflt}

\usepackage{dsfont}

\newcommand{\FF}{\ensuremath{\mathbb F}}

\newcommand\X{\mathcal{X}}
\newcommand\Y{\mathcal{Y}}
\newcommand\Z{\mathcal{Z}}

\DeclareMathOperator*{\rank}{rank}
\DeclareMathOperator*{\spann}{span}

\author{Pavel Gubkin}

\address{
\begin{flushleft}
Pavel Gubkin: pasha\_gubkin\_v@mail.ru\\\vspace{0.1cm}
St. Petersburg State University \\
Universitetskaya nab. 7-9, St. Petersburg, 199034, Russia
\end{flushleft}
}

\thanks{The work is supported by Ministry of Science and Higher Education of the Russian Federation, agreement № 075–15–2019–1619.}
\title{On unique tensor rank decomposition of 3-tensors}
\begin{document}
\maketitle
\begin{abstract}
    We answer to a question posed recently in
    \cite{lovitz2021}, proving the conjectured sufficient minimality and uniqueness condition of the 3-tensor decomposition. 
\end{abstract}
\section{Introduction}
Let $\FF$ be a field, $m$ be a positive integer and $V_1$,\ldots, $V_m$ be vector spaces over $\FF$. We say that tensor rank $\rank(v)$ of  $v\in V_1 \otimes V_2\otimes\ldots\otimes V_m$ is the minimal number such that $v$ can be expressed as a sum of $n$ product tensors, i.e., tensors of the form $v_1\otimes v_2,\ldots\otimes v_m$. Also for a collection of vectors $u_1,\ldots, u_r$  $k$-rank is defined as the largest number $k$ such that for any $k$ vectors  $u_{i_1},\ldots,u_{i_k}$ we have $\dim\spann(u_{i_1},\ldots,u_{i_k}) = k$.

In 1977 J. Kruskal proved \cite{Kruskal1977} that a tensor of a form
\begin{gather*}
    \sum_{i = 1}^n x_i\otimes y_i\otimes z_i
\end{gather*}
constitutes a unique representation as a sum of $n$ product tensors if $k$-ranks of the sets $\{x_1,\ldots, x_n\}$, $\{y_1,\ldots, y_n\}$ and $\{z_1,\ldots, z_n\}$ are large enough. Proof of this theorem in more general form can be found in \cite{Rhodes2010}. There are some other conditions that  imply uniqueness of the tensor rank decompositions in the literature. In this paper we focus on so-called U-condition which was originally introduces in \cite{Domanov201321}, \cite{Domanov20132}; we will use the definition from \cite{lovitz2021}.

Denote $[n]=\{1,\ldots,n\}$; $[n..m]=\{n,\ldots,m\}$.
By $\langle A\rangle$ denote the linear span of a set 
$A$ in a linear space. We also write $\langle x_1,\ldots,x_k\rangle$ meaning
$\langle \{x_1,\ldots,x_k\}\rangle$. For a sequence $\alpha=(\alpha_1,\ldots,\alpha_n)\in \FF^n$ let $\omega(\alpha)$ be the number of indices $i$ for which $\alpha_i\ne 0$. We say that a set of tensors  \{$x_i\otimes y_i\otimes z_i$\}, $i\in[n]$ satisfies the \textbf{Condition U} if for any $\alpha\in \FF^n$
\begin{gather}\label{U-condition}
    \rank\left(\sum_{i = 1}^n \alpha_i y_i\otimes z_i \right)\ge \min(\omega(\alpha), n - d + 2),
\end{gather}
where $d$ is the dimension of $\langle x_i|i\in [n]\rangle$. The aim of this paper is Question 38 from \cite{lovitz2021}, we prove the following
\begin{theorem}\label{theorem1} Let $\FF$ be a field; $\X, \Y, \Z$ be 
linear spaces over $\FF$; $x_1,\ldots,x_n\in \X$;
$y_1,\ldots,y_n\in \Y$; $z_1,\ldots,z_n\in \Z$. Assume that
\begin{itemize}
    \item $x_i\ne 0$ for all $i$ and also $x_i$ and $x_j$ are not parallel for all $i\ne j$;
    \item the set of tensors  \{$x_i\otimes y_i\otimes z_i$\}, $i\in[n]$ satisfies the Condition U.
\end{itemize}
Then for $\theta:=\sum_{i=1}^n x_i\otimes y_i\otimes z_i$ we have $\rank \theta=n$ 
and the representation of $\theta$ as a sum of $n$ product tensors is unique.
\end{theorem}
\begin{rem}
The theorem holds true if we demand \eqref{U-condition} not for all $\alpha\in F^n$, but only for $\alpha$ such that $\alpha =(\eta(x_1),\ldots, \eta(x_n)$ for some linear functional $\eta\in \X^*$. In other words, following the notations from \cite{lovitz2021}, the theorem holds with Condition U replaced by condition W.
\end{rem}

\section{Proof of the Theorem \ref{theorem1}}

We will use the following

\begin{lemma}[Lemma 12.5.3.5, \cite{Landsberg2012}]\label{l1}
Let $W$ be a $d$-dimensional linear space over a field $\FF$, and $\{x_i\}$, $\{v_i\}$, $i\in[n]$ be two sets of vectors such that
\begin{itemize}
    \item $\langle v_i| i\in [n]\rangle = W$;
    \item $x_i$ and $x_j$ are not parallel for any two distinct indices $i,j$;
    \item for any hyperplane $H$
    \begin{gather*}
        \#(H\cap \{v_i\}_{i\in [n]})\ge d - 1 \Longrightarrow  \#(H\cap \{x_i\}_{i\in [n]})\ge  \#(H\cap \{v_i\}_{i\in [n]}).
    \end{gather*}
\end{itemize}
Then there exists a permutation $\pi$ of indices $[n]$ and the non-zero scalars $\lambda_i\in \FF$ such that $v_i=\lambda_i x_{\pi_i}$ for all $i\in [n]$.
\end{lemma}
\begin{proof}
Let us outline the idea of the proof. Proceeding by induction one has to prove 
\begin{itemize}
    \item for any $m$ - dimensional subspace $H_m$
    \begin{gather*}
        \#(H\cap \{v_i\}_{i\in [n]})\ge m \Longrightarrow  \#(H\cap \{x_i\}_{i\in [n]})\ge  \#(H\cap \{v_i\}_{i\in [n]}).
    \end{gather*}
\end{itemize}
Notice that this assertion for $m = d - 1$ is given in the lemma and for $m = 1$ this immediately implies the lemma. For details see Section 12.5.4, \cite{Landsberg2012}.
\end{proof}

\begin{proof}[Proof of the Theorem \ref{theorem1}]
Assume that $\theta=\sum_{i=1}^n v_i\otimes u_i \otimes w_i$ for certain $v_i\in \X$, $u_i\in \Y$,
$w_i\in \Z$. We need to  show that the multisets $\{x_i\otimes y_i\otimes z_i|i\in [n]\}$ and
$\{v_i\otimes u_i\otimes w_i|i\in [n]\}$ coincide. We do it in two steps. At the first step, we 
apply Lemma \ref{l1} to prove that

$(\heartsuit)$ there exists a permutation $\pi$ of indices $[n]$ and the non-zero
scalars $\lambda_i\in \FF$ such that $v_i=\lambda_i x_{\pi_i}$ for all $i\in [n]$.

In order to prove $(\heartsuit)$, take an arbitrary linear functional $\eta\in \X^*$.
We get
\begin{equation}\label{eq1}
\sum_{i=1}^n \eta(x_i) y_i\otimes z_i=\sum_{i=1}^n \eta(v_i) u_i\otimes w_i=: \Theta(\eta).
\end{equation}
Notice that for $\eta\in \langle v_i | i\in [n]\rangle^\perp$ one has $\Theta(\eta) = 0$ and the Condition U implies that $\eta(x_i) = 0$ for every $i\in[n]$. Therefore, $$\langle x_i | i\in [n]\rangle\subset \langle v_i | i\in [n]\rangle=: W$$ 
and without loss of generality we can now assume $\X = W$.  

For a hyperplane $H$ denote the orthogonal projection onto $H^\perp$ by  $\eta_H$. Consider hyperplane $H$ such that 
\begin{gather}\label{hyperplane condition}
    \#(H\cap \{v_i\}_{i\in [n]})\ge \dim W - 1.
\end{gather}
To apply Lemma 3 we need to show that
\begin{gather*}
    \#(H\cap \{x_i\}_{i\in [n]})\ge  \#(H\cap \{v_i\}_{i\in [n]}),
\end{gather*}
or equivalently 
\begin{gather*}
    \#\{i | \eta_H(x_i) = 0\}\ge \#\{i | \eta_H(v_i) = 0\}
\end{gather*}
Denote the two latter sets by $I$ and $J$ respectively
then
$$
\rank \Theta(\eta_H)\leqslant n-|J|\leqslant n-\dim W+1\leqslant n - d + 1,
$$
where $d$ is the dimension of  $\langle x_i | i\in [n]\rangle$. On the other hand, from the condition U we kwow
\begin{gather*}
    \rank \Theta(\eta_H)\geqslant \min(|\{i:\eta_H(x_i)\ne 0\}|,n-d+2).
\end{gather*}
If the latter minimum equals $n-d+2$, this
is a contradiction. Thus $$n-|J|\geqslant \rank \Theta(\eta_H)\geqslant |\{i:\eta_H(x_i)\ne 0\}|=n-|I|.$$
Therefore $|I|\ge |J|$ for any hyperplane $H$ satisfying \eqref{hyperplane condition} and we get $(\heartsuit)$ by Lemma \ref{l1}.

Now, permuting indices and using $(\lambda x)\otimes u\otimes v=x\otimes (\lambda u)\otimes v$
we may suppose that $v_i=x_i$ for all $i$. Without loss of generality we may suppose that
$x_1,\ldots,x_d$ is a basis of $\langle x_i| i\in [n]\rangle$. Denote $x_i=\sum_{j=1}^d \alpha_{ij} x_j$
for $i\in [d+1..n]$. We get
$$
T_j:=y_j\otimes z_j+\sum_{i=d+1}^n \alpha_{ij} y_i\otimes z_i=u_j\otimes w_j+\sum_{i=d+1}^n \alpha_{ij} u_i\otimes w_i, \quad j=1,\ldots,d.
$$
By Condition U, we get
\begin{gather*}
    \rank T_j=1+|\{i:\alpha_{ij}\ne 0\}|.
\end{gather*}
Denote 
\begin{gather*}
    J_j=\{j\}\cup \{i:\alpha_{ij}\ne 0\}.
\end{gather*}
By the above maximal rank condition we get that the vectors $\{y_i:i\in J_j\}$ are linearly independent. Analogously, so are the vectors
$\{z_i:i\in J_j\}$, $\{u_i:i\in J_j\}$, $\{v_i:i\in J_j\}$. Next, assume that
$k\in J_j$ but $y_k\notin \langle u_i:i\in J_j\rangle$. Then there exists a projector $P\colon \Y\to  \langle u_i:i\in J_j\rangle$ 
for which $Py_k=0$. We get $(P\otimes Id) T_j=T_j$ since $Pu_i=u_i$ for all $i\in J_j$, but
\begin{gather*}
    (P\otimes Id) T_j=(Py_j)\otimes z_j+\sum_{i>d} \alpha_{ij} (Py_i)\otimes z_i
\end{gather*}
is a representation of 
$T_j$ as a sum of less than $|J_j|$ product tensors. Thus all $y_k$ with $k\in J_j$ belong to 
the subspace $\langle u_i:i\in J_j\rangle=:\Y_j$. Since $|J_j|\leqslant n-d+1$, and any $n-d+1$ (and even any $n-d+2$) $y_j$'s
are linearly independent by Condition U, we get $\Y_j=\langle y_i:i\in J_j\rangle$. Next, $\dim (\Y_j+\Y_k)=|J_j\cup J_k|$
whenever $j,k\in [d]$ and $j\ne k$, that's because $|J_j\cup J_k|\leqslant n-d+2$ and corresponding $y_i$'s with $i\in J_j\cup J_k$
are linearly independent. But $\Y_j+\Y_k$ is also generated by $u_i$'s with $i\in J_j\cup J_k$. Thus, these $u_i$'s are also linearly independent.

Let $s\in [d+1..n]$. Since no two $x_i$'s are parallel, there exist at least two indices $j,k \in [d]$ 
for which $\alpha_{sj}\ne 0$ and $\alpha_{sk}\ne 0$. In other words, $s\in J_j\cap J_k$. Then, $y_s$ belongs
to $\Y_j\cap \Y_k$. This space is generated by $u_i$ with $i\in J_j\cap J_k\subset [d+1..n]$. So, we have just
proved that
\begin{gather*}
    y_s=\sum_{i=d+1}^n f_{si} u_i
\end{gather*}
for certain scalars $f_{si}$. Now fix $s,k \in [d+1..n]$
such that $s\ne k$. Since $x_s$ and $x_k$ are not parallel, we may without loss of generality assume that $\alpha_{s1} \alpha_{k2}\ne \alpha_{k1}\alpha_{s2}$.
Consider the tensor $\alpha_{k2} T_1-\alpha_{k1}T_2$. It equals
$$
p y_{1}\otimes z_1+q y_2\otimes z_2+\sum_{i=d+1}^n r_i y_i\otimes z_i=
p u_{1}\otimes w_1+q u_2\otimes w_2+\sum_{i=d+1}^n r_i u_i\otimes w_i
$$
for certain coefficients $p,q,r_{d+1},\ldots,r_n$, where $r_i=\alpha_{k2} \alpha_{i1}-\alpha_{k1} \alpha_{i2}$. In particular, $r_k=0$ but
$r_s\ne 0$. As above, we conclude that $y_s$ belongs to a span of $u_1,u_2$ and $\{u_i: r_i\ne 0\}$.  Note that $U:=\{u_1,u_2$, $u_{d+1}, \ldots,u_n\}$
are linearly independent, since $y_1,y_2,y_{d+1},\ldots,y_n$ lie in their span (and any $n-d+2$ vectors $y_i$'s are linearly independent).
Therefore, the expansion of $y_s$ in the system $U$ does not contain $u_k$. This yields $f_{sk}=0$. 

So, we proved that $u_i$'s are parallel to $y_i$'s for $i>d$. Analogously, $w_i$ are parallel to $z_i$ for $i>d$, and we get $u_i\otimes w_i=\lambda_i y_i\otimes z_i$
for $i>d$. Fix $i>d$ and find $j\in [d]$ for which $\alpha_{ij}\ne 0$. Looking at $T_j$ we get 
$$
y_j\otimes z_j+\sum_{i=d+1}^n (1-\lambda_i) \alpha_{ij} y_i\otimes z_i=u_j\otimes w_j.
$$ 
If there are at least two non-zero coefficients in the left hand side, this contradicts to condition (ii). Therefore $\lambda_i=1$. This holds for all $i=d+1,\ldots,n$. So, $u_i\otimes w_i=y_i\otimes z_i$
for all $i>d$. Looking at two expressions for $T_j$ ($j=1,\ldots,d$) we get the same for $j\leqslant d$.
\end{proof}

\section{Field dependence of the Condition U}
It is well known that tensor rank is field dependent (see section 3.1 in \cite{kolda2009}, or section 2 in \cite{Bergman1969}). Here we show that U-condition is also field dependent.

\medskip

\noindent \textbf{Example.}
Consider three tensors over $\mathbb{F}_2$:
\begin{gather*}
    a\otimes a,\quad b\otimes b, \quad (a + b)\otimes (a + b).
\end{gather*}
Linear combinations over the field $\mathbb{F}_2$ are just sums of subsets, and one can check that
\begin{gather}\label{f2}
    \rank(\alpha_1 a\otimes a + \alpha_2 b\otimes b + \alpha_3(a + b)\otimes (a + b)) = \min(\omega(\alpha),2 ).
\end{gather}
On the other hand, we can consider the same set of tensors over $\FF_4 = \FF_2[\xi]$, where $\xi^2 = \xi + 1$. We have
\begin{gather}
    \nonumber
    (1 + \xi)a\otimes a + b\otimes b +  \xi(a + b)\otimes (a + b)
    =
    \\ \nonumber
    =a\otimes a + (1 + \xi)b\otimes b + \xi a\otimes b + \xi b\otimes a =
    \\ \label{f4}
    = (a + \xi b)\otimes(a + \xi b).
\end{gather}
Now consider 3 following 3-dimensional tensors:
\begin{gather}\label{3tensors}
    a\otimes a\otimes a,\quad b\otimes b\otimes b, \quad c\otimes(a + b)\otimes (a + b).
\end{gather}
We have $n = d = 3$. Hence \eqref{f2} and \eqref{f4} imply that condition U holds over $\FF_2$ and does not hold over $\FF_4$. Also from Theorem \ref{theorem1} and the statement below it follows that the uniqueness of tensor decomposition of sum of tensors in \eqref{3tensors} holds over $\FF_2$ and does not hold over $\FF_4$.

\begin{statement}
If $n = d$, then the Condition U is not only sufficient for uniqueness of the tensor rank decomposition, but also a necessary one.
\end{statement}
\begin{proof}
Indeed, for $n = d$ Condition U takes the form
\begin{gather*}
    \rank\left(\sum_{i = 1}^n \alpha_i y_i\otimes z_i \right)\ge \min(\omega(\alpha), 2).
\end{gather*}
Assume that Condition U does not hold. Then there exists $\alpha$ such that $\omega(\alpha)\ge 2$ and 
\begin{gather*}
    \rank\left(\sum_{i = 1}^n \alpha_i y_i\otimes z_i \right)\le 1,
\end{gather*}
equivalently
\begin{gather*}
    \sum_{i = 1}^n \alpha_i y_i\otimes z_i  = y\otimes z
\end{gather*}
for some vectors $y\in\Y$ and $z\in Z$. Without loss of generality, assume that $\alpha_1\neq 0$, then 
\begin{gather*}
    y_1\otimes z_1 = \alpha_1^{-1}\left(y\otimes z -\sum_{i = 1}^n \alpha_i y_i\otimes z_i \right)
\end{gather*}
and 
\begin{gather*}
    \sum_{i = 1}^n x_i\otimes y_i\otimes z_i = \alpha_1^{-1}x_1\otimes\left(y\otimes z -\sum_{i = 1}^n \alpha_i y_i\otimes z_i \right) + \sum_{i = 2}^n x_i\otimes y_i\otimes z_i =
    \\
    = \alpha_1^{-1}x_1\otimes y\otimes z + \sum_{i = 2}^n (x_i - \alpha_1^{-1}\alpha_i x_1)\otimes y_i\otimes z_i,
\end{gather*}
therefore the uniqueness does not hold.
\end{proof}

\section{Acknowledgements}

I am grateful to B. Lovitz for pointing my attention
to Lemma \ref{l1}, which allowed to simplify the argument
in the case of finite field dramatically. 


\bibliographystyle{plain}
\bibliography{references.bib}

\end{document}